\documentclass[10pt]{article}
\usepackage{amsmath}
\usepackage[english,  activeacute]{babel}
\usepackage[latin1]{inputenc}
\usepackage{amssymb}
\usepackage{amsthm}
\usepackage{pgf,tikz,pgfplots}
\usepackage{mathrsfs}
\usetikzlibrary{arrows}
\usepackage{graphics, graphicx}
\usepackage{array}
\usepackage{a4wide}
\usepackage{color, url}
\usepackage{float}
\usepackage{pgf,tikz,pgfplots}
\usepackage{slantsc}
\usetikzlibrary{arrows}
\usepackage[affil-it]{authblk}

\usetikzlibrary{matrix,arrows,calc}
\theoremstyle{plain}
\newtheorem{theorem}{Theorem}

\newtheorem{lemma}[theorem]{Lemma}

\theoremstyle{definition}
\newtheorem{definition}[theorem]{Definition}
\newtheorem{example}[theorem]{Example}
\newtheorem{remark}[theorem]{Remark}

\newcommand{\N}{{\mathbb N}}

\newcommand{\Z}{{\mathbb Z}}

\newcommand{\D}{{\mathcal D}}

\newcommand{\T}{{\mathcal T}}

\pgfdeclarelayer{overlay}

\title{ Asymmetric extension of Pascal-Dellanoy triangles}

\author{Said Amrouche, Hac\`{e}ne Belbachir\\
USTHB, Faculty of Mathematics, RECITS Laboratory, BP 32, El Alia, 16111, Bab Ezzouar,
Algiers, Algeria\\
saidamrouchee@gmail.com, hbelbachir@usthb.dz  
}

\date{}
\begin{document}
\maketitle
\begin{abstract}
We give a generalization of the Pascal triangle called the quasi s-Pascal triangle where the sum of the elements crossing the diagonal rays produce the s-bonacci sequence. For this, consider a lattice path in the plane whose step set is $\lbrace L=(1,0), L_{1}=(1,1), L_{2}=(2,1),\ldots, L_{s}=(s,1)\rbrace $; an explicit formula is given. Thereby linking  the elements of the quasi s-Pascal triangle with the bi$^{s}$nomial coefficients. We establish the recurrence relation for the sum of elements lying over any finite ray of the quasi s-Pascal triangle. The generating function of the cited sums is produced. We also give identities among which one equivalent to the de Moivre sum and establish a q-analogue of the coefficient of the quasi s-Pascal triangle. 
\end{abstract}\section{Introduction}

A lattice path in the plane-$(x,y)$ is a set of edges $\lbrace{p_0},p_{1},\ldots,p_{n}\rbrace$ in $\Z^{2}$, such that two edges are related by one vertex, the set of vertices connecting $p_{0}$ to $p_{n}$ is called a lattice path. Several authors studied and enumerated the lattice path, Mohanty and Handa \cite{MH} enumerate the unrestricted lattice paths from $(0,0)$ to $(n,k)$ where $u$ diagonal steps are allowed at each position, Dziemianczuk \cite{D} count the lattice path with four steps horizontal $H=(1,0)$, vertical $V=(0,1)$, diagonal $D=(1,1)$, and sloping $L=(-1,1)$,  Fray and Roselle \cite{FR} determine the number of unrestricted weighted lattice paths from $(0,0)$ to $(n,k)$ with horizontal, vertical, and diagonal steps, Rohatgi \cite{R} enumerate the paths which must remain below the line $y = ax + b$ where the diagonal steps are allowed in addition to the usual horizontal and vertical steps. In a Pascal triangle, the binomial coefficients  $\binom {n}{k}$ count the number of lattice paths from $(0,0)$ to $(n,k)$ using the steps $\lbrace H=(1,0)\rightarrow, D=(1,1) \nearrow \rbrace$, see {\scriptsize{SLOANE}} \cite{S1} $A007318$.\\
It is well known that the terms of  Fibonacci sequence $(F_{n})_{n}$ are obtained by summing the elements crossing the principal diagonal rays in the Pascal triangle,
\begin{align}
F_{n+1}=\sum_{k}\binom{n-k}{k},
\end{align}
where $\binom{n}{k}=\frac{n!}{k!(n-k)!}$ for $n\geq k \geq 0$ and $\binom{n}{k}=0$ otherwise.\\
The generating function of the binomial coefficients is given by
\begin{align}
\sum_{n\geq0}\binom{n}{k}x^{n}=\frac{x^{k}}{(1-x)^{k+1}}.
\end{align}
Alladi and Hoggat \cite{tribo1}  extended the Pascal triangle. They established the Tribonacci triangle and proved that the sum of elements lying over the principal diagonal rays in the Tribonacci triangle gives the Tribonacci sequence 
\begin{align*}
T_{n+1}=T_{n}+T_{n-1}+T_{n-2},
\end{align*}
with $T_{0}=0, T_{1}=1, T_{2}=1$.\\
Denote by $\binom{n}{k}_{[2]}$ the element in the $n^{th}$ row and $k^{th}$ column of the Tribonacci triangle, the triangle is produced by the recurrence relation
\begin{align*}
\binom{n}{k}_{[2]}=\binom{n-1}{k}_{[2]}+\binom{n-1}{k-1}_{[2]}+\binom{n-2}{k-1}_{[2]},
\end{align*}
where  $\binom{n}{0}_{[2]}$= $\binom{n}{n}_{[2]}=1.$ We use the convention $\binom{n}{k}_{[2]}=0$ for $k\notin
\left \{ 0,\ldots ,n\right \}.$\\
Moreover, Barry \cite{PB} has shown that for $0\leq k\leq n$ these coefficients satisfy the relation
\begin{align}
\binom{n}{k}_{[2]}=\sum_{j=0}^{k}\binom kj \binom{n-j}{k}.
\end{align}
\[
\begin{tabular}{|c|ccccccccccc}
\hline
$_{\mathbf{n}\backslash \mathbf{k}}$ & \textbf{0} & \textbf{1} & \textbf{2}
& \textbf{3} & \textbf{4} & \textbf{5} & \textbf{6} & \textbf{7} & \textbf{8}
& \textbf{9} \\ \hline
\textbf{0} & 1\  &  &  &  &  &  &  &  &  &  &\\ 
\textbf{1} & 1 & 1 &  &  &  &  &  &  &  & & \\ 
\textbf{2} & 1 & 3 & 1 &  &  &  &  &  &  &  &\\ 
\textbf{3} & 1 & 5 & 5 & 1 &  &  &  &  &  &  &\\ 
\textbf{4} & 1 & 7 & 13 & 7 & 1 &  &  &  &  &  & \\ 
\textbf{5} & 1 & 9& 25 & 25 & 9 & 1 &  &  &  & & \\ 
\textbf{6} & 1 & 11 & 41 & 63 & 41 & 11 & 1 &  &  &  &\\ 
\textbf{7} & 1 & 13 & 61 & 129 &129 & 61 & 13 & 1 &  &  &\\ 
\textbf{8} & 1 & 15 & 85 & 231 & 321 & 231 & 85 & 15 & 1 &  &\\ 
\textbf{9} & 1 & 17 & 113 & 377 & 681 & 681 & 377 & 113 & 17 & 1 &
\end{tabular}
\]
\begin{center}
\text{Table 1. Tribonacci triangle.\smallskip} 
\end{center}
We find in {\scriptsize{SLOANE}} \cite{S1}  A008288 that $\binom{n}{k}_{[2]}$ counts the number of lattice paths from $(0,0)$ to $(n,k)$ using the steps $\lbrace H=(1,0), D=(1,1), L=(2,1)  \rbrace$.\\
In what follows $s$ is a positive integer.
\subsection{ The s-Pascal triangle}  Let  $ k \in \lbrace 0,1,\ldots, sn \rbrace$, the bi$^{s}$nomial coefficient $\binom{n}{k}_{s}$ is defined as the $k^{th}$ coefficient in the expansion\\
$$(1+x+x^{2}+\cdots+x^{s})^{n}=\sum_{k \in \Z}\binom{n}{k}_{s}x^{k},$$
with $\binom{n}{k}_{s}=0$ for $k >sn$ or $k < 0$.\\
Using the classical binomial coefficient, see \cite{A,HBB,B}, one has\\
\begin{align}
\binom{n}{k}_{s}=\sum_{j_{1}+j_{2}+\cdots+j_{s}=k}\binom{n}{j_{1}}\binom{j_{1}}{j_{2}}\cdots\binom{j_{s-1}}{j_{s}}.
\end{align}
Some other readily well known established properties are:\\
the symmetry relation\\
\begin{align}
\binom{n}{k}_{s}=\binom{n}{sn-k}_{s},
\end{align}
the longitudinal recurrence relation
\begin{align}
\binom{n}{k}_{s}=\sum_{j=0}^{s}\binom{n-1}{k-j}_{s},
\end{align}
 the diagonal recurrence relation
\begin{align}
\binom{n}{k}_{s}=\sum_{j=0}^{n}\binom{n}{j}\binom{j}{k-j}_{s-1},
\end{align}
and de Moivre expression see \cite{AD,AD1}
 \begin{align}\label{js}
\binom{n}{k}_{s}=\sum_{j=0}(-1)^{j}\binom{n}{j}\binom{k-j(s+1)+n-1}{n-1}.
\end{align}
Belbachir and Benmezai \cite{bb} gave the following relation
\begin{align}\label{bb1}
\binom{n}{k}_{s}=(-1)^{k}\sum_{j_{1}+j_{2}+\cdots+j_{s}=k}\binom{n}{j_{1}}\binom{n}{j_{2}}\cdots\binom{n}{j_{s}}a^{-\sum_{r=1}^{s}rj_{r}},
\end{align}
where $a=\exp{(2i\pi/(s+1))}$.\\
These coefficients, as for the usual binomial coefficients, are built, as for the Pascal triangle, the
called  s-Pascal  triangle. One can find the first values of the s-Pascal  triangle in {\scriptsize{SLOANE}}   $A027907$ for $s = 2$,  $A008287$ for $s = 3$ and  $A035343$ for $s = 4$.\\
\
\[
\begin{tabular}{|c|ccccccccccccccc}
\hline
$_{\mathbf{n}\backslash \mathbf{k}}$ & \textbf{0} & \textbf{1} & \textbf{2}
& \textbf{3} & \textbf{4} & \textbf{5} & \textbf{6} & \textbf{7} & \textbf{8} & \textbf{9} & \textbf{10}& \textbf{11}& \textbf{12}& \textbf{13}
\\ \hline
  \textbf{0} & 1\  &  &  &  &  &  &  &  &  &  &  & & &  \\ 
\textbf{1} & 1 & 1 & 1 & 1  &  &  &  &  &  &  &  &  & &  \\ 
\textbf{2} & 1 & 2 & 3 & 4 & 3 & 2 & 1 &  &  &  &  &  &     \\ 
\textbf{3} & 1 & 3 & 6 & 10 & 12 & 12& 10 & 6 & 3 & 1 &  &  &  &   \\ 
\textbf{4} & 1 & 4 & 10 & 20 &31 & 40 & 44 & 40 &31  & 20 & 10& 4 & 1 &  \\ 
\textbf{5} & 1 & 5 & 15 & 35 &65 & 101 & 135 & 155 &155  & 135 & 101& 65 & 35 & 15 \\ 
\end{tabular}
\]
\begin{align*}
\text{Table 2. Biquadranomial triangle  ($s=3$).\smallskip} 
\end{align*}
\subsection{The $q$-binomial coefficient}  The $q$-analogue of binomial coefficient or  the $q$-binomial coefficient
$[^{n} _{k}]$ generalizes the binomial coefficient $\binom{n}{k}$ \cite{ht,JC}. It is defined as follows 
\begin{align*}
{n \brack k}=\frac{[n]_{q}!}{[k]_{q}![n-k]_{q}!}q^{\binom k2},
\end{align*}
with $ [n]_{q}= 1+q+q^{2}+\cdots+q^{n-1}$ and $[n]_{q}! =[1]_{q}[2]_{q}\cdots[n]_{q}$, we use the convention ${n \brack k}=0$ for $k \notin \lbrace 0,\ldots, n \rbrace$.\\
The $q$-binomial coefficient satisfies the following recurrence relations
\begin{align}\label{03}
{n \brack k}={n-1 \brack k}+q^{n-1}{n-1 \brack k-1},
\end{align}
and
\begin{align}\label{04}
{n \brack k}=q^{k}{n-1 \brack k}+q^{k-1}{n-1 \brack k-1}.
\end{align}

And  the generating functions are
\begin{align}\label{1}
\sum_{k=0}^{n}{n \brack k}x^{k}=(1+x)(1+qx)(1+q^{2}x)\cdots(1+q^{n-1}x),
\end{align}
and 
\begin{align}\label{2t}
\sum_{n\geq0}{n \brack k}x^{n}=\frac{x^{k}q^{\binom{k}{2}}}{(1-x)(1-qx)\cdots(1-q^{k}x)}.
\end{align}
Belbachir and Benmezai \cite{bb} proposed the  $q$-bi$^{s}$nomial coefficient, denoted by ${n \brack k}^{(s)}$, as follows 
\begin{align}\label{10}
{n \brack k}^{(s)}:=(-1)^{k}\sum_{j_{1}+j_{2}+\cdots+j_{s}=k}{n \brack j_{1}}{n \brack j_{2}}\cdots{n \brack j_{s}}a^{-\sum_{r=1}^{s}rj_{r}}.
\end{align}
The $q$-bi$^{s}$nomial coefficient satisfies the  following recurrence relations
\begin{align}
{n \brack k}^{(s)}=\sum_{j=0}^{s}q^{k-j}{n-1 \brack k-j}^{(s)},
\end{align}
\begin{align}
{n \brack k}^{(s)}=\sum_{j=0}^{s}q^{(n-1)j}{n-1 \brack k-j}^{(s)}.
\end{align}
According to (\ref{1}) the generating function is given by
\begin{align}\label{gh}
\sum_{k=0}^{ns}{n \brack k}^{(s)}x^{k}=\prod_{j=0}^{n-1}(1+q^{j}x+(q^{j}x)^{2}+\cdots+(q^{j}x)^{s}).
\end{align}
In the first section we introduce the quasi $s$-Pascal  triangle by using a family of lattice paths; we establish an explicit formula for the elements of the quasi $s$-Pascal triangle, and we prove that the sums of the elements crossing the diagonal rays yield the terms of   $s$-bonacci sequence; we close this section by giving a relation between $s$-Pascal triangle and quasi $s$-Pascal  triangle. The second section is devoted to the linear recurrence relation obtained by summing the elements lying over any finite rays of the quasi $s$-Pascal triangle and we give the corresponding generating function. In the third section we give the de Moivre like summation with some other identities. In section four, we establish the $q$-analogue of the elements of the quasi $s$-Pascal triangle. 

\section{The quasi s-Pascal triangle }
In this section we define the quasi $s$-Pascal triangle, we denote by $\binom{n}{k}_{[s]}$  the coefficient in the $n^{th}$ row and $k^{th}$ column of the quasi $s$-Pascal triangle.
\begin{definition}\label{fd} The quasi-bi$^{s}$nomial coefficient $\binom{n}{k}_{[s]}$ is defined by the number of lattice path from $(0,0)$ to $(n,k)$ with steps in $\lbrace L=(1,0), L_{1}=(1,1), L_{2}=(2,1),\ldots, L_{s}=(s,1)\rbrace$. With $\binom{n}{0}_{[s]}=\binom{n}{n}_{[s]}=1 \ \text{and the convention} \ \binom{n}{k}_{[s]}=0 \ \text{for} \ k > n $  or $k<0.$ 
 \end{definition}
\begin{center}
\definecolor{ududff}{rgb}{0.30196078431372547,0.30196078431372547,1.}
\begin{tikzpicture}[line cap=round,line join=round,>=triangle 45,x=1.0cm,y=1.0cm]
\clip(-3.5,-3) rectangle (17,6);
\draw [->,line width=0.4pt] (0.52,5.) -- (1.52,5.6);
\draw [->,line width=0.4pt] (2.52,5.) -- (3.98,5.58);
\draw [->,line width=0.4pt] (6.,5.) -- (8.52,5.56);
\draw [->,line width=0.4pt] (-1.,5.) -- (-0.34,5.58);
\draw [->,line width=0.4pt] (-2.64,5.) -- (-2.,5.);
\begin{scriptsize}
\draw [fill=ududff] (-2.64,5.) circle (1.5pt);
\draw [fill=ududff] (-1.,5.) circle (1.5pt);
\draw [fill=ududff] (-0.32,4.98) circle (1.5pt);
\draw [fill=ududff] (-0.98,5.58) circle (1.5pt);
\draw [fill=ududff] (0.52,5.) circle (1.5pt);
\draw [fill=ududff] (1.,5.) circle (1.5pt);
\draw [fill=ududff] (1.5,4.98) circle (1.5pt);
\draw [fill=ududff] (1.52,5.6) circle (1.5pt);
\draw [fill=ududff] (1.04,5.58) circle (1.5pt);
\draw [fill=ududff] (0.52,5.56) circle (1.5pt);
\draw [fill=ududff] (2.52,5.) circle (1.5pt);
\draw [fill=ududff] (2.54,5.56) circle (1.5pt);
\draw [fill=ududff] (3.04,5.56) circle (1.5pt);
\draw [fill=ududff] (3.,5.) circle (1.5pt);
\draw [fill=ududff] (3.48,5.) circle (1.5pt);
\draw [fill=ududff] (3.5,5.56) circle (1.5pt);
\draw [fill=ududff] (3.98,5.58) circle (1.5pt);
\draw [fill=ududff] (4.,5.) circle (1.5pt);
\draw [fill=ududff] (6.,5.) circle (1.5pt);
\draw [fill=ududff] (6.52,4.98) circle (1.5pt);
\draw [fill=ududff] (8.,5.) circle (1.5pt);
\draw [fill=ududff] (8.52,5.) circle (1.5pt);
\draw [fill=ududff] (8.52,5.56) circle (1.5pt);
\draw [fill=ududff] (8.04,5.56) circle (1.5pt);
\draw [fill=ududff] (6.5,5.56) circle (1.5pt);
\draw [fill=ududff] (5.98,5.54) circle (1.5pt);
\draw [fill=ududff] (7.34,5.54) circle (0.5pt);
\draw [fill=ududff] (7.,5.54) circle (0.5pt);
\draw [fill=ududff] (7.64,5.54) circle (0.5pt);
\draw [fill=ududff] (7.62,5.) circle (0.5pt);
\draw [fill=ududff] (7.36,5.) circle (0.5pt);
\draw [fill=ududff] (7.,5.) circle (0.5pt);
\draw [fill=ududff] (4.6,4.98) circle (0.5pt);
\draw [fill=ududff] (5.,5.) circle (0.5pt);
\draw [fill=ududff] (5.4,5.) circle (0.5pt);
\draw [fill=ududff] (-0.34,5.58) circle (1.5pt);
\draw [fill=ududff] (-2.,5.) circle (1.5pt);
\large\draw (-2.7,4.5) node[anchor=north west] { $\text{L}$};
\large\draw (-1,4.5) node[anchor=north west] { $\text{$L_{1}$}$};
\large\draw (.8,4.5) node[anchor=north west] { $\text{$L_{2}$}$};
\large\draw (3,4.5) node[anchor=north west] { $\text{$L_{3}$}$};
\large\draw (7,4.5) node[anchor=north west] { $\text{$L_{s}$}$};
\end{scriptsize}
\end{tikzpicture}
\vspace*{-7cm}
\end{center}
\begin{center}
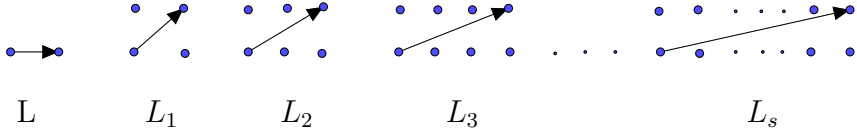
Ilustration of possible steps for the coefficient $\binom{n}{k}_{[s]}.$
\end{center}
\
\begin{lemma}
 The quasi-bi$^{s}$nomial coefficient $\binom{n}{k}_{[s]}$ satisfies the following recurrence relation
\begin{align}\label{aabc}
\binom{n}{k}_{[s]}=\binom{n-1}{k}_{[s]}+\binom{n-1}{k-1}_{[s]}+\binom{n-2}{k-1}_{[s]}+\cdots+\binom{n-s}{k-1}_{[s]},
\end{align}
\end{lemma}
\begin{proof}  By Definition \ref{fd}, the last step of any path is one of $\lbrace L=(1,0), L_{1}=(1,1), L_{2}=(2,1),\ldots, L_{s}=(s,1)\rbrace$, then if the last step is $L=(1,0)$, it remains to enumerate the number of lattice paths from $(0,0)$ to $(n-1,k)$ which is $\binom{n-1}{k}_{[s]}$, or if the last one is $L_{1}=(1,1)$ it remains to enumerate the number of lattice paths from $(0,0)$ to $(n-1,k-1)$ which is $\binom{n-1}{k-1}_{[s]},\ldots,$ if the last step is $L_{s}=(s,1)$ it remains to enumerate the number of lattice paths from $(0,0)$ to $(n-s,k-1)$ which is $\binom{n-s}{k-1}_{[s]}$, considering all possibilities we construct our recurrence.
\end{proof}
\subsection{Generating function}
Here is given the generating function of $\lbrace \binom{n}{k}_{[s]} \rbrace_{n}$;
\begin{theorem}\label{aqw} Let $F_{k}(x):=\sum_{n\geq 0}\binom{n}{k}_{[s]} x^{n}$, then
 $$F_{k}(x)=(1+x+x^{2}+\cdots+x^{s-1})^{k}\frac{x^{k}}{(1-x)^{k+1}}.$$
\end{theorem}
\begin{proof}
It follows from Relation  (\ref{aabc}) that
\begin{align*}
F_{k}(x)=xF_{k}(x)+xF_{k-1}(x)+x^{2}F_{k-1}(x)+\cdots+x^{s}F_{k-1}(x),
\end{align*}
repeated applications of this recurrence give the result. 
\end{proof}
\subsection{Binomial coefficients explicit formula}
The following result gives an explicit formula for the coefficients of the quasi $s$-Pascal triangle in terms of binomial coefficients and a variant with multinomial coefficients.
\begin{theorem}\label{po} The explicit formula for the quasi-bi$^{s}$nomial coefficient is given by
\begin{align}\label{T1}
\binom{n}{k}_{[s]}=\sum_{j_{1}}\sum_{j_{2}}\cdots\sum_{j_{s-1}}\binom {k}{j_{1}}\binom {j_{1}}{j_{2}}\cdots \binom{j_{s-2}}{j_{s-1}}\binom{n-\sum_{i=1}^{s-1}j_{i}}{k},
\end{align}
the multinomial version is
\begin{align}\label{T2}
\binom{n}{k}_{[s]}&=\sum_{k_{1},k_{2},\ldots,k_{s}}\binom {k}{k_{1},k_{2},\ldots,k_{s}}\binom {n+k-\sum_{i=1}^{s}ik_{i}}{k},
\end{align}
where $ \binom {k}{k_{1},k_{2},\ldots,k_{s}}=\frac{k!}{k_{1}!k_{2}!\cdots k_{s}!}$ for $k_{1}+k_{2}+,\cdots+k_{s}=k$ and $ \binom {k}{k_{1},k_{2},\ldots,k_{s}}=0$, else.
\end{theorem}
\begin{proof} 
For  Relation $(\ref{T1})$ we need to prove that 
\begin{align*}
&\sum_{n=0}^{\infty}\sum_{j_{1}}\sum_{j_{2}}\cdots\sum_{j_{s-1}}\binom {k}{j_{1}}\binom {j_{1}}{j_{2}}\cdots \binom{j_{s-2}}{j_{s-1}}\binom{n-\sum_{i=1}^{s-1}j_{i}}{k}x^{n}=\left(\frac{x+x^{2}+\cdots+x^{s}}{1-x}\right)^{k}\frac{1}{1-x}.\\
&\text{So}\\ 
&\sum_{n=0}^{\infty}\sum_{j_{1}}\sum_{j_{2}}\cdots\sum_{j_{s-1}}\binom {k}{j_{1}}\binom {j_{1}}{j_{2}}\cdots \binom{j_{s-2}}{j_{s-1}}\binom{n-\sum_{i=1}^{s-1}j_{i}}{k}x^{n}\\
&=\sum_{j_{1}}\binom {k}{j_{1}}x^{j_{1}}\sum_{j_{2}}\binom {j_{1}}{j_{2}}x^{j_{2}}\cdots\sum_{j_{s-1}}\binom{j_{s-2}}{j_{s-1}}x^{j_{s-1}}\sum_{n=0}^{\infty}\binom{n-\sum_{i=1}^{s-1}j_{i}}{k}x^{n-\sum_{i=1}^{s-1}j_{i}}\\
&=\frac{x^{k}}{(1-x)^{k+1}}\sum_{j_{1}}\binom {k}{j_{1}}x^{j_{1}}\sum_{j_{2}}\binom {j_{1}}{j_{2}}x^{j_{2}}\cdots\sum_{j_{s-2}}\binom{j_{s-3}}{j_{s-2}}x^{j_{s-2}}\sum_{j_{s-1}}\binom{j_{s-2}}{j_{s-1}}x^{j_{s-1}}\\
 &=\frac{x^{k}}{(1-x)^{k+1}}\sum_{j_{1}}\binom {k}{j_{1}}x^{j_{1}}\sum_{j_{2}}\binom {j_{1}}{j_{2}}x^{j_{2}}\cdots\sum_{j_{s-2}}\binom{j_{s-3}}{j_{s-2}}(x+x^{2})^{j_{s-2}}\\
\vdots \hspace*{-0.3cm}\\
&=\frac{x^{k}}{(1-x)^{k+1}}\sum_{j_{1}}\binom{k}{j_{1}}(x+x^{2}+\cdots+x^{s-1})^{j_{1}}=\left(\frac{x+x^{2}+\cdots+x^{s}}{1-x}\right)^{k}\frac{1}{1-x}.
\end{align*}
For  Relation $(\ref{T2})$ we have
\begin{align*}
&\sum_{n\geq0}\sum_{k_{1},k_{2},\ldots,k_{s}}\binom {k}{k_{1},k_{2},\ldots,k_{s}}\binom {n-\sum_{i=1}^{s}ik_{i}+k}{k}x^{n}\\
&=\frac{1}{x^{k}}\sum_{k_{1},k_{2},\ldots,k_{s}}\binom {k}{k_{1},k_{2},\ldots,k_{s}}x^{\sum_{i=1}^{s}ik_{i}}\sum_{n\geq0}\binom {n-\sum_{i=1}^{s}ik_{i}+k}{k}x^{n-\sum_{i=1}^{s}ik_{i}+k}\\
&=\frac{1}{(1-x)^{k+1}}\sum_{k_{1},k_{2},\ldots,k_{s}}\binom {k}{k_{1},k_{2},\ldots,k_{s}}x^{\sum_{i=1}^{s}ik_{i}}\\
&=\left(\frac{x+x^{2}+\cdots+x^{s}}{1-x}\right)^{k}\frac{1}{1-x}.
\end{align*}
\end{proof}
\subsection{Link with generalized Dellanoy matrix}
Ramirez and Sirvent \cite{RS} propose a generalization of Dellanoy and Pascal Riordan arrays, they denoted by $\D_{m}(n,k)$ the element in the $n^{th}$ row and $k^{th}$ column of the generalized Dellanoy matrix, such that $\D_{m}(n,k)$ satisfy the following recurrence relation 
\begin{align*}
\D_{m}(n+1,k)=a\D_{m}(n+1,k-1)+\sum_{i=0}^{m-1}a_{i+1}\D_{m}(n,k-i),
\end{align*}
with $k\geq m-1, n\geq 1$ and initial conditions $\D_{m}(0,k)=a^{k}$ and $\D_{m}(n,0)=a_{1}^{n}$.\\
The coefficient $\D_{m}(n,k)$ is given by the following explicit formula
\begin{align*}
\D_{m}(n,k)=\sum_{j_{1},j_{2},\ldots,j_{m-1}}\binom{n}{j_{1}}\binom{n-j_{1}}{j_{2}}\cdots\binom{n-j_{1}-\cdots-j_{m-2}}{j_{m-1}}\binom{n+k-u}{n}\times\\
\times a_{1}^{j_{1}}a_{2}^{j_{2}}\cdots a_{m-1}^{j_{m-1}}a_{m}^{n-\sum_{i=1}^{m-1}j_{i}}a^{k-u},
\end{align*}
where $u=(m-1)(n-j_{1})+\sum_{i=2}^{m-1}(i-m)j_{i}.$\\

The following result gives a relation between $\D_{m}(n,k)$ and $\binom{n}{k}_{[s]}$.
\begin{theorem}For $m=s$, $a=1$ and $a_{i}=1$, $i \in \lbrace 1,\ldots,s \rbrace$
\begin{align*}
\D_{s}(k,n-k)=\binom{n}{k}_{[s]}.
\end{align*}
\end{theorem}
\begin{proof}  For $a=1$ and $a_{i}=1$, $i \in \lbrace 1,\ldots,s \rbrace$ we have
\begin{align*}
\D_{s}(n,k)=\sum_{j_{1},j_{2},\ldots,j_{s-1}}\binom{n}{j_{1}}\binom{n-j_{1}}{j_{2}}\cdots\binom{n-j_{1}-\cdots-j_{s-2}}{j_{s-1}}\binom{n+k-u}{n},\hspace*{5cm}
\end{align*}
where $u=(s-1)(n-j_{1})+\sum_{i=2}^{s-1}(i-s)j_{i},$ then
\begin{align*}
\D_{s}(n,k)=\sum_{j_{1},j_{2},\ldots,j_{s-1}}\binom{n}{n-j_{1}}\binom{n-j_{1}}{n-j_{1}-j_{2}}\cdots\binom{n-j_{1}-\cdots-j_{s-2}}{n-j_{1}-\cdots-j_{s-2}-j_{s-1}}\binom{n+k-u}{n},\hspace*{1cm}
\end{align*}
we put $ j_{v}^{'}\rightarrow n-\sum_{l=1}^{v}j_{l}, \ v=\lbrace 1,\ldots,s-1 \rbrace  $  then $ j_{1}^{'}+j_{2}^{'}+\cdots+j_{s-1}^{'}=u$,
\begin{align*}
\D_{s}(n,k)=\sum_{j_{1}^{'},j_{2}^{'},\ldots,j_{s-1}^{'}}\binom {n}{j_{1}^{'}}\binom {j_{1}^{'}}{j_{2}^{'}}\cdots \binom{j_{s-2}^{'}}{j_{s-1}^{'}}\binom{n+k-\sum_{i=1}^{s-1}j_{i}^{'}}{n}.\hspace*{6.5cm}\\
\text{Thus} \ \D_{s}(k,n-k)=\binom{n}{k}_{[s]}.\hspace*{13cm}
\end{align*}
\end{proof}
\subsection{Recurrence relation or s-bonacci sequence}
Now we establish the recurrence relation for the $s$-bonacci sequence, which is a generalization of Fibonacci sequence. Let $(T_{n,s})_{n}$ be the terms of the  $s$-bonacci sequence obtained by summing the elements lying over the  principal diagonal rays in the quasi $s$-Pascal triangle.\\
Let be the sequence
\begin{align*}
T_{n+1,s}:=\sum_{k}\binom{n-k}{k}_{[s]},  
\end{align*}
with $T_{0,s}=0$.
\begin{theorem}\label{aqwd} For  $ n \geq 0$, $(T_{n,s})_{n} $  satisfies the following recurrence relation
\[
T_{n+1,s}=T_{n,s}+T_{n-1,s}+\cdots+T_{n-s,s},
\]
with $T_{1,s}=1,T_{-i,s}=0 \ \text{for} \ i \in \lbrace 0,-1,\ldots,-(s-1) \rbrace$.
\end{theorem}
It is not else $s$-bonacci sequence.
\begin{proof}  We have $T_{n+1,s}=\sum_{k}\binom{n-k}{k}_{[s]}$ and by Relation (\ref{aabc}) we obtain\\
$
T_{n+1,s}=\sum_{k}\binom{n-k-1}{k}_{[s]}+\sum_{k}\binom{n-k-1}{k-1}_{[s]}+\cdots+\sum_{k}\binom{n-k-s}{k-1}_{[s]}\\
\hspace*{4.2cm} k^{'}\rightarrow k-1 \\
\hspace*{0.8cm}=\sum_{k}\binom{n-k-1}{k}_{[s]}+\sum_{k^{'}}\binom{n-k^{'}-2}{k^{'}}_{[s]}+\cdots+\sum_{k^{'}}\binom{n-k^{'}-s-1}{k^{'}}_{[s]}\\
\\
\hspace*{0.8cm}=T_{n,s}+T_{n-1,s}+\cdots+T_{n-s,s}.
$
\end{proof}
For $s=1$ and $s=2$ we obtain the terms of Fibonacci and Tribonacci sequences respectively.
\begin{example} For $s=3$ we have the quadrabonacci triangle 
\[
\begin{tabular}{|c|cccccccccc}
\hline
$_{\mathbf{n}\backslash \mathbf{k}}$ & \textbf{0} & \textbf{1} & \textbf{2}
& \textbf{3} & \textbf{4} & \textbf{5} & \textbf{6} & \textbf{7} & \textbf{8}
& \textbf{9} \\ \hline
\textbf{0} & 1\  &  &  &  &  &  &  &  &  &  \\ 
\textbf{1} & 1 & 1 &  &  &  &  &  &  &  &  \\ 
\textbf{2} & 1 & 3 & 1 &  &  &  &  &  &  &  \\ 
\textbf{3} & 1 & 6 & 5 & 1 &  &  &  &  &  &  \\ 
\textbf{4} & 1 & 9 & 15 & 7 & 1 &  &  &  &  &  \\ 
\textbf{5} & 1 & 12& 33 & 28 & 9 & 1 &  &  &  &  \\ 
\textbf{6} & 1 & 15 & 60 & 81 & 45 & 11 & 1 &  &  &  \\ 
\textbf{7} & 1 & 18 & 96 & 189 &66 & 33 & 13 & 1 &  &  \\ 
\textbf{8} & 1 & 21 & 141 & 378 & 459 & 281 & 91 & 15 & 1 &  \\ 
\textbf{9} & 1 & 24 & 195 & 675 & 1107 & 946 & 449 & 120 & 17 & 1%
\end{tabular}
\]
\vspace*{-0.52cm}
\begin{center}
\text{Table 2. The  quadrabonacci triangle.\smallskip} 
\end{center}
By Relation (\ref{aabc}) the elements of the  quadrabonacci triangle ($s=3$) are given by $\binom{n}{0}_{[3]}=\binom{n}{n}_{[3]}=1$, and
$$\binom{n}{k}_{[3]}=\binom{n-1}{k}_{[3]}+\binom{n-1}{k-1}_{[3]}+\binom{n-2}{k-1}_{[3]}+\binom{n-3}{k-1}_{[3]}.$$
For $s=3$,  $\binom{n}{k}_{[3]}$ counts  the number of lattice paths with steps in $(1,0)$, $(1,1)$, $(2,1)$, $(3,1)$ from $(0,0)$ to $(n,k)$, for example for the value  $6$ in the quadrabonacci triangle we have the  lattice path
\vspace*{-0.64cm}
\begin{center}
\definecolor{ududff}{rgb}{0.30196078431372547,0.30196078431372547,1.}
\begin{tikzpicture}[line cap=round,line join=round,>=triangle 45,x=1.0cm,y=1.0cm]
\clip(-7.8,-2.7) rectangle (16.24,7.8);
\draw [->,line width=0.2pt] (-6.,6.) -- (-5.,7.);
\draw [->,line width=0.2pt] (-5.,7.) -- (-4.,7.);
\draw [->,line width=0.2pt] (-4.,7.) -- (-3.,7.);
\draw [->,line width=0.2pt] (-2.,6.) -- (-1.,6.);
\draw [->,line width=0.2pt] (-1.,6.) -- (0.,7.);
\draw [->,line width=0.2pt] (0.,7.) -- (1.,7.);
\draw [->,line width=0.2pt] (2.,6.) -- (3.,6.);
\draw [->,line width=0.2pt] (3.,6.) -- (4.,6.);
\draw [->,line width=0.2pt] (4.,6.) -- (5.,7.);
\draw [->,line width=0.2pt] (-6.,4.) -- (-5.,4.);
\draw [->,line width=0.2pt] (-5.,4.) -- (-3.,5.);
\draw [->,line width=0.2pt] (-2.,4.) -- (0.,5.);
\draw [->,line width=0.2pt] (0.,5.) -- (1.,5.);
\draw [->,line width=0.2pt] (2.,4.) -- (5.,5.);
\begin{scriptsize}
\draw [fill=ududff] (-6.,7.) circle (1.5pt);
\draw [fill=ududff] (-5.,7.) circle (1.5pt);
\draw [fill=ududff] (-4.,7.) circle (1.5pt);
\draw [fill=ududff] (-3.,7.) circle (1.5pt);
\draw [fill=ududff] (-6.,6.) circle (1.5pt);
\draw [fill=ududff] (-5.,6.) circle (1.5pt);
\draw [fill=ududff] (-4.,6.) circle (1.5pt);
\draw [fill=ududff] (-3.,6.) circle (1.5pt);
\draw [fill=ududff] (-2.,7.) circle (1.5pt);
\draw [fill=ududff] (-2.,6.) circle (1.5pt);
\draw [fill=ududff] (-1.,7.) circle (1.5pt);
\draw [fill=ududff] (-1.,6.) circle (1.5pt);
\draw [fill=ududff] (0.,7.) circle (1.5pt);
\draw [fill=ududff] (0.,6.) circle (1.5pt);
\draw [fill=ududff] (1.,7.) circle (1.5pt);
\draw [fill=ududff] (1.,6.) circle (1.5pt);
\draw [fill=ududff] (2.,7.) circle (1.5pt);
\draw [fill=ududff] (2.,6.) circle (1.5pt);
\draw [fill=ududff] (3.,7.) circle (1.5pt);
\draw [fill=ududff] (3.,6.) circle (1.5pt);
\draw [fill=ududff] (4.,7.) circle (1.5pt);
\draw [fill=ududff] (5.,7.) circle (1.5pt);
\draw [fill=ududff] (5.,6.) circle (1.5pt);
\draw [fill=ududff] (4.,6.) circle (1.5pt);
\draw [fill=ududff] (-6.,5.) circle (1.5pt);
\draw [fill=ududff] (-6.,4.) circle (1.5pt);
\draw [fill=ududff] (-5.,4.) circle (1.5pt);
\draw [fill=ududff] (-5.,5.) circle (1.5pt);
\draw [fill=ududff] (-4.,4.) circle (1.5pt);
\draw [fill=ududff] (-4.,5.) circle (1.5pt);
\draw [fill=ududff] (-3.,4.) circle (1.5pt);
\draw [fill=ududff] (-3.,5.) circle (1.5pt);
\draw [fill=ududff] (-2.,4.) circle (1.5pt);
\draw [fill=ududff] (-2.,5.) circle (1.5pt);
\draw [fill=ududff] (-1.,4.) circle (1.5pt);
\draw [fill=ududff] (-1.,5.) circle (1.5pt);
\draw [fill=ududff] (0.,5.) circle (1.5pt);
\draw [fill=ududff] (0.,4.) circle (1.5pt);
\draw [fill=ududff] (1.,4.) circle (1.5pt);
\draw [fill=ududff] (1.,5.) circle (1.5pt);
\draw [fill=ududff] (2.,4.) circle (1.5pt);
\draw [fill=ududff] (2.,5.) circle (1.5pt);
\draw [fill=ududff] (3.,5.) circle (1.5pt);
\draw [fill=ududff] (3.,4.) circle (1.5pt);
\draw [fill=ududff] (4.,5.) circle (1.5pt);
\draw [fill=ududff] (4.,4.) circle (1.5pt);
\draw [fill=ududff] (5.,5.) circle (1.5pt);
\draw [fill=ududff] (5.,4.) circle (1.5pt);
\end{scriptsize}
\end{tikzpicture}
\end{center}
\vspace*{-5.8cm}
\begin{center}
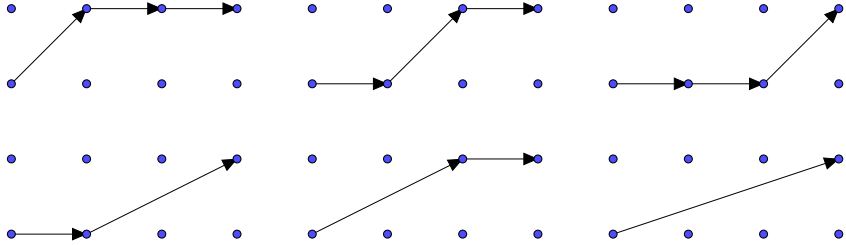
Title: Illustration of possible paths from $(0,0)$ to $(3,1)$ using the steps\\
$(1,0), (1,1), (2,1), (3,1).$
\end{center}
\end{example}
Notice that for $s=1$ and $s=2$ the obtained triangles are symmetric, unlike the cases where $s > 2$.\\
\subsection{$s$-Pascal triangle versus quasi-$s$-Pascal triangle}
The following result establishes the relation between the quasi $s$-Pascal triangle and $s$-Pascal  triangle 
\begin{theorem}\label{azer} For fixed non negative integers $n, k$ and $s$, we have
\begin{align*}
\binom{n}{k}_{[s]}=\sum_{i}\binom{n-i}{k}\binom{k}{i}_{s-1}.
\end{align*}
\end{theorem}

\begin{proof} We have 
\begin{align*}
\binom{n}{k}_{[s]}=\sum_{j_{1}}\sum_{j_{2}}\cdots\sum_{j_{s-1}}\binom {k}{j_{1}}\binom {j_{1}}{j_{2}}\cdots \binom{j_{s-2}}{j_{s-1}}\binom{n-j_{1}-j_{2}-\cdots-j_{s-1}}{k},\hspace*{1.5cm}\\
\text{considering the summations by blocks} \ j_{1}+j_{2}+\cdots+j_{s-1}=i \ \text{we get}\hspace*{2cm}\\
\binom{n}{k}_{[s]}=\sum_{i}\binom{n-i}{k}\sum_{j_{1}+j_{2}+\cdots+j_{s-1}=i}\binom {k}{j_{1}}\binom {j_{1}}{j_{2}}\cdots \binom{j_{s-2}}{j_{s-1}}=\sum_{i}\binom{n-i}{k}\binom{k}{i}_{s-1}.\hspace*{0.3cm}\\
\end{align*}
\end{proof}
\vspace*{-1.95cm}
\section{Linear recurrence relation and generating function associated to finite transversals of the quasi s-Pascal triangle}
This section is devoted to establish a recurrence relation associated to the sums of the elements lying over  the transversals of direction $(\alpha,r)$ in the quasi $s$-Pascal triangle. In \cite{bzk} we find the details about the concept  of direction in Pascal triangle. The study was extended for the arithmetic triangle, see \cite{ABR,BS1}. We generalize the concept to our triangle as an extension of Theorem \ref{aqwd} to the case where $r\in \Z,\ \alpha \in \N$, $\beta \in\Z^{+}$ with $0\leq \beta<\alpha$ and $r+\alpha>0$. This corresponds to the finite sequences lying over finite transversals of the quasi $s$-Pascal triangle.\\
Let be the sequence
\begin{align*}
 T_{n+1,s}^{(\alpha,\beta,r)}:=\sum_{k}\binom{n-rk}{\beta +\alpha k}_{[s]}, \ \text{with} \ T_{0,s}^{(\alpha,\beta,r)}=0.
\end{align*}
The following figure illustrate the direction $(\alpha,r)=(2,1)$ and  $\beta=0$ on the Tribonacci triangle.
\begin{center}
\begin{tikzpicture}[line cap=round,line join=round,>=triangle 45,x=1.0cm,y=1.0cm]
\clip(-6.405887445887451,6.667012987013018) rectangle (10.78891774891774,16.121558441558467);
\draw (3.8971428571428497,10.545800865800894) node[anchor=north west] {$1$};
\draw (-3.1331601731601793,15.446233766233792) node[anchor=north west] {$1$};
\draw (-3.1331601731601793,14.788225108225136) node[anchor=north west] {$1$};
\draw (-0.120173160173167,13.420259740259768) node[anchor=north west] {$1$};
\draw (-3.1331601731601793,14.060952380952408) node[anchor=north west] {${1}$};
\draw (-2.111515151515158,14.095584415584444) node[anchor=north west] {${3}$};
\draw (-3.115844155844162,13.437575757575786) node[anchor=north west] {${1}$};
\draw (-2.128831168831175,14.82285714285717) node[anchor=north west] {${1}$};
\draw (-2.111515151515158,13.437575757575786) node[anchor=north west] {${5}$};
\draw (-1.0898701298701365,13.437575757575786) node[anchor=north west] {${5}$};
\draw (2.910129870129863,11.325021645021673) node[anchor=north west] {${1}$};
\draw (-1.1937662337662405,12.814199134199162) node[anchor=north west] {${13}$};
\draw (-0.120173160173167,12.848831168831197) node[anchor=north west] {${7}$};
\draw (0.8841558441558371,12.866147186147215) node[anchor=north west] {${1}$};
\draw (-2.111515151515158,12.796883116883144) node[anchor=north west] {${7}$};
\draw (-3.115844155844162,12.76225108225111) node[anchor=north west] {${1}$};
\draw (-1.2110822510822576,12.069610389610418) node[anchor=north west] {${25}$};
\draw (-2.128831168831175,12.069610389610418) node[anchor=north west] {${9}$};
\draw (0.8668398268398199,12.069610389610418) node[anchor=north west] {${9}$};
\draw (-0.25,12.069610389610418) node[anchor=north west] {${25}$};
\draw (-3.1331601731601793,12.034978354978383) node[anchor=north west] {${1}$};
\draw (1.9058008658008585,12.069610389610418) node[anchor=north west] {${1}$};
\draw (1.8365367965367894,10.545800865800894) node[anchor=north west] {${61}$};
\draw (0.7975757575757506,11.307705627705657) node[anchor=north west] {${41}$};
\draw (0.7283116883116814,10.563116883116912) node[anchor=north west] {${129}$};
\draw (-3.115844155844162,11.290389610389639) node[anchor=north west] {${1}$};
\draw (-1.21,10.58043290043293) node[anchor=north west] {${61}$};
\draw (-0.33,10.563116883116912) node[anchor=north west] {${129}$};
\draw (1.819220779220772,11.325021645021673) node[anchor=north west] {${11}$};
\draw (-0.25,11.325021645021673) node[anchor=north west] {${63}$};
\draw (-1.21,11.325021645021673) node[anchor=north west] {${41}$};
\draw (-2.1980952380952443,11.307705627705657) node[anchor=north west] {${11}$};
\draw (-3.098528138528145,10.563116883116912) node[anchor=north west] {${1}$};
\draw (-1.1071861471861537,14.095584415584444) node[anchor=north west] {${1}$};
\draw (-2.1980952380952443,10.597748917748946) node[anchor=north west] {${13}$};
\draw (2.806233766233759,10.545800865800894) node[anchor=north west] {${13}$};
\draw [->,line width=0.1pt] (-2.7522077922077974,15.22112554112557) -- (-1.1764502164502222,15.879134199134228);
\draw [->,line width=0.1pt] (-2.7695238095238146,14.511168831168861) -- (-1.1071861471861526,15.221125541125572);
\draw [->,line width=0.1pt] (-0.7,13.85)--(0.7695238095238146,14.511168831168861)  ;
\draw [->,line width=0.1pt] (1.25,12.55)--(2.5,13.1)  ;
\draw [->,line width=0.1pt] (-2.8387878787878833,13.80121212121212) -- (-1.141818181818187,14.511168831168861);
\draw [->,line width=0.1pt] (-0.7435497835497893,13.177835497835527) -- (0.7975757575757516,13.853160173160203);
\draw [->,line width=0.1pt] (-0.65,12.5) -- (0.8841558441558381,13.14);
\draw [->,line width=0.1pt] (-2.734891774891781,12.5) -- (-1.1,13.1);
\draw [->,line width=0.1pt] (-2.769523809523816,11.740606060606092) -- (-1.2457142857142922,12.346666666666698);

\draw [->,line width=0.1pt] (-2.821471861471868,11.065281385281416) -- (-1.228398268398275,11.671341991342022);
\draw [->,line width=0.1pt] (-0.65,11.85) -- (0.8495238095238024,12.415930735930766);
\draw [->,line width=0.1pt] (-2.8041558441558503,13.10857142857146) -- (-1.124502164502171,13.731948051948084);
\draw (-3.115844155844162,9.87047619047622) node[anchor=north west] {${1}$};
\draw (-2.180779220779227,9.905108225108254) node[anchor=north west] {${15}$};
\draw (-1.21,9.922424242424272) node[anchor=north west] {${85}$};
\draw (-0.3,9.922424242424272) node[anchor=north west] {${231}$};
\draw (0.72,9.905108225108254) node[anchor=north west] {${321}$};
\draw (1.75,9.922424242424272) node[anchor=north west] {${231}$};
\draw (3.87,9.9) node[anchor=north west] {${85}$};
\draw (2.82,9.922424242424272) node[anchor=north west] {${15}$};
\draw (4.884155844155837,9.887792207792238) node[anchor=north west] {${1}$};
\draw [->,line width=0.1pt] (-2.7522077922077983,10.268744588744617) -- (-1.2457142857142922,10.857489177489207);
\draw [->,line width=0.1pt] (-0.7089177489177556,11.08259740259743) -- (0.8148917748917679,11.67134199134202);
\draw [->,line width=0.1pt] (-2.7002597402597464,9.61073593073596) -- (-1.228398268398275,10.13021645021648);
\draw [->,line width=0.1pt] (-0.7089177489177556,10.32069264069267) -- (0.7283116883116814,10.857489177489207);
\draw [->,line width=0.1pt] (1.35,11.05) -- (2.875497835497828,11.67134199134202);
\draw [->,line width=0.1pt] (1.2304761904761834,11.80987012987016) -- (2.754285714285707,12.38129870129873);
\end{tikzpicture}
\end{center}
\vspace*{-3cm}
\begin{center}
Figure 1: Tribonacci triangle.
\end{center}
\begin{theorem}\label{s} For $n \geq \alpha s+r$, $ (T_{n+1,s}^{(\alpha,\beta,r)})_{n}$ satisfies the following linear recurrence
relation
\begin{align}\label{asc}
\sum^{\alpha}_{i=0}(-1)^{i}\binom{\alpha}{i}T_{n-i,s}^{(\alpha,\beta,r)}=\sum^{\alpha(s-1)}_{i=0}\binom{\alpha}{i}_{s-1}T_{n-\alpha-r-i,s}^{(\alpha,\beta,r)},
\end{align}
we can recover the initial conditions $\T_{1,s}^{(\alpha,\beta,r)},\ldots,\T_{\alpha s+r-1,s}^{(\alpha,\beta,r)}$ by $ \sum_{k}\binom{n-rk}{\beta +\alpha k}_{[s]} $.
\end{theorem}
For the proof we need the following Lemma.
\begin{lemma}[\cite{bzk}] Let $a$,$b$ and $\alpha$ be non negative integers satisfying the conditions $\alpha\leq a$, then
\begin{align}\label{asc1}
\sum^{\alpha}_{i=0}(-1)^{i}\binom{\alpha}{i}\binom{a-i}{b}=\binom{a-\alpha}{b-\alpha}.
\end{align}
\end{lemma}
\begin{proof} Of Theorem \ref{s},
from Relations (\ref{T1}) and (\ref{asc1}), we get
\begin{align*}
&\sum^{\alpha}_{i=0}(-1)^{i}\binom{\alpha}{i}T_{n-i,s}^{(\alpha,\beta,r)} \\
&=\sum^{\alpha}_{i=0}(-1)^{i}\binom{\alpha}{i}\sum_{k}\binom{n-rk-i-1}{\beta +\alpha k}_{[s]}\\
&=\sum^{\alpha}_{i=0}(-1)^{i}\binom{\alpha}{i}\sum_{k}\sum_{j_{1},j_{2},\ldots, j_{s-1}}\binom {\beta+\alpha k}{j_{1}}\binom {j_{1}}{j_{2}}\ldots \binom{j_{s-2}}{j_{s-1}}\binom{n-rk-i-\sum_{i=1}^{s-1}j_{i}-1}{\beta+\alpha k}\\
&=\sum_{k}\sum_{j_{1},j_{2},\ldots, j_{s-1}}\binom {\beta+\alpha k}{j_{1}}\binom {j_{1}}{j_{2}}\cdots \binom{j_{s-2}}{j_{s-1}}\sum^{\alpha}_{i=0}(-1)^{i}\binom{\alpha}{i}\binom{n-rk-i-\sum_{i=1}^{s-1}j_{i}-1}{\beta+\alpha k}\\
&=\sum_{k}\sum_{j_{1},j_{2},\ldots, j_{s-1}}\binom {\beta+\alpha k}{j_{1}}\binom {j_{1}}{j_{2}}\cdots \binom{j_{s-2}}{j_{s-1}}\binom{n-rk-\alpha-\sum_{i=1}^{s-1}j_{i}-1}{\beta+\alpha(k-1)}\\
&\quad (k^{'}\rightarrow k-1)\\
&=\sum_{k^{'}}\sum_{j_{1},j_{2},\ldots, j_{s-1}}\binom {\beta+\alpha k^{'}+\alpha}{j_{1}}\binom {j_{1}}{j_{2}}\cdots \binom{j_{s-2}}{j_{s-1}}\binom{n-rk^{'}-r-\alpha-\sum_{i=1}^{s-1}j_{i}-1}{\beta+\alpha k^{'}}\\
&=\sum_{k^{'}}\sum_{j_{1},j_{2},\ldots, j_{s-1}}\binom {\beta+\alpha k^{'}+\alpha}{j_{1}}\binom {j_{1}-i_{1}+i_{1}}{j_{2}}\cdots \binom{j_{s-2}-i_{s-2}+i_{s-2}}{j_{s-1}}\times\\
&\times\binom{n-rk^{'}-r-\alpha-\sum_{i=1}^{s-1}j_{i}-1}{\beta+\alpha k^{'}}\\
&\text{by Vandermonde Formula}\\
&=\sum_{k^{'}}\sum_{j_{1},j_{2},\ldots, j_{s-1}}\sum_{i_{1},i_{2},\ldots, i_{s-1}}\binom {\alpha}{i_{1}}\binom {\beta+\alpha k^{'}}{j_{1}-i_{1}}\binom {i_{1}}{i_{2}}\binom{j_{1}-i_{1}}{j_{2}-i_{2}}\cdots \binom{i_{s-2}}{i_{s-1}}\binom{j_{s-2}-i_{s-2}}{j_{s-1}-i_{s-1}}\times\\
&\times\binom{n-rk^{'}-r-\alpha-\sum_{i=1}^{s-1}j_{i}-1}{\beta+\alpha k^{'}}\\
& (l_{v}\rightarrow j_{v}-i_{v})\\
&=\sum_{k^{'}}\sum_{l_{1},l_{2},\ldots, l_{s-1}}\sum_{i_{1},i_{2},\ldots, i_{s-1}}\binom {\alpha}{i_{1}}\binom {\beta+\alpha k^{'}}{l_{1}}\binom {i_{1}}{i_{2}}\binom{l_{1}}{l_{2}}\cdots \binom{i_{s-2}}{i_{s-1}}\binom{l_{s-2}}{l_{s-1}}\times\\
&\times\binom{n-rk^{'}-r-\alpha-\sum_{j=1}^{s-1}i_{j}-\sum_{j=1}^{s-1}l_{j}-1}{\beta+\alpha k^{'}}\\
&\text{we take the summation as by block} \ i_{1}+i_{2}+\cdots+i_{s-1}=i \\
\end{align*}
\begin{align*}
&=\sum_{k^{'}}\sum_{i}\sum_{l_{1},l_{2},\ldots, l_{s-1}}\sum_{i_{1}+i_{2}+\cdots +i_{s-1}=i}\binom {\alpha}{i_{1}}\binom {\beta+\alpha k^{'}}{l_{1}}\binom {i_{1}}{i_{2}}\binom{l_{1}}{l_{2}}\cdots \binom{i_{s-2}}{i_{s-1}}\binom{l_{s-2}}{l_{s-1}}\times\\
&\times\binom{n-rk^{'}-r-\alpha-i-\sum_{j=1}^{s-1}l_{j}-1}{\beta+\alpha k^{'}}\\
&=\sum_{k^{'}}\sum_{i}\sum_{l_{1},l_{2},\ldots, l_{s-1}}\binom {\beta+\alpha k^{'}}{l_{1}}\binom{l_{1}}{l_{2}}\cdots \binom{l_{s-2}}{l_{s-1}}\binom{n-rk^{'}-r-\alpha-i-\sum_{j=1}^{s-1}l_{i}-1}{\beta+\alpha k^{'}}\times\\
&\times\sum_{i}\binom {\alpha}{i_{1}}\binom {i_{1}}{i_{2}}\binom {i_{1}}{i_{2}}\cdots\binom{i_{s-2}}{i_{s-1}}\hspace*{7cm}\\
&=\sum^{\alpha(s-1)}_{i=0}\binom{\alpha}{i}_{s-1}\T_{n-\alpha-r-i,s}^{(\alpha,\beta,r)}.
\end{align*}
\end{proof}
For $\alpha=1,r=1,\beta=0$ we obtain the terms of  $s$-bonacci sequence, for $s=2$, we obtain  Theorem $7$ of \cite{AB}. \\
\begin{example}
For $\alpha=2,r=1,\beta=0$ and $n\geq 2s+1$ we have the following recurrence relation
\begin{align*}
T_{n,s}^{(2,0,1)}&=\sum_{j=1}^{2}(-1)^{j+1}\binom{2}{j}T_{n-j,s}^{(2,0,1)}+\sum_{j=0}^{2(s-1)}\binom{2}{j}_{s-1}T_{n-j-3,s}^{(2,0,1)}\\
&=2T_{n-1,s}^{(2,0,1)}-T_{n-2,s}^{(2,0,1)}+T_{n-3,s}^{(2,0,1)}+2T_{n-4,s}^{(2,0,1)}+\cdots\\
&\cdots+sT_{n-s-2,s}^{(2,0,1)}+\cdots+2T_{n-2s,s}^{(2,0,1)}+T_{n-2s-1,s}^{(2,0,1)},
\end{align*}
as $\binom{2}{s-1}_{s-1}=s$.
\end{example}
The following result establishes the generating function for the sequence $(\T_{n,s}^{(\alpha,\beta,r)})_{n}$ for quasi $s$-Pascal triangles.
\begin{theorem} The generating function of the sequence $\lbrace\T_{n,s}^{(\alpha,\beta,r)}\rbrace_{n\geq0}$ is given by
\begin{align*}
\sum_{n\geqslant0}\T_{n+1,s}^{(\alpha,\beta,r)}x^{n}=\frac{(1-x)^{\alpha-\beta-1}(x+x^{2}+\cdots+x^{s})^{\beta}}{(1-x)^{\alpha}-x^{r+\alpha}(1+x+\cdots+x^{s-1})^{\alpha}}.
\end{align*}
\end{theorem}
\begin{proof} We have
\begin{align*}
\sum_{n\geqslant0}\T_{n+1,s}^{(\alpha,\beta,r)}x^{n}&=\sum_{n\geqslant0}\sum_{k}\binom{n-rk}{\beta +\alpha k}_{[s]}x^{n}\\
&=\sum_{n\geqslant qk}\sum_{k}\binom{n-rk}{\beta +\alpha k}_{[s]}x^{n-rk}x^{rk}\\
&=\sum_{k}\frac{(x+x^{2}+\cdots+x^{s})^{\beta+\alpha k}x^{rk}}{(1-x)^{\beta+\alpha k+1}}\\
&=\frac{(x+x^{2}+\cdots+x^{s})^{\beta}}{(1-x)^{\beta+1}}\sum_{k}\left( \frac{(x+x^{2}+\cdots+x^{s})^{\alpha}x^{r}}{(1-x)^{\alpha}}\right)^{k}\\
&=\frac{(x+x^{2}+\cdots+x^{s})^{\beta}}{(1-x)^{\beta+1}}\frac{1}{1-(\frac{x^{r}(x+x^{2}+\cdots+x^{s})^{\alpha}}{(1-x)^{\alpha}})}\\
&=\frac{(1-x)^{\alpha-\beta-1}(x+x^{2}+\cdots+x^{s})^{\beta}}{(1-x)^{\alpha}-x^{r+\alpha}(1+x+\cdots+x^{s-1})^{\alpha}}.
\end{align*}
\end{proof}
\section{The de Moivre summation and other nested sums}
Butler and Karasik see \cite{BK} showed how the binomial coefficient can be written as nested sums, in this section we establish an identity for the quasi bi$^{s}$nomial coefficients $\binom{n}{k}_{[s]}$ equivalent to the de Moivre summation for bi$^{s}$nomial coefficient and we give some other nested sums for the coefficient $\binom{n}{k}_{[s]}$. The following identity is important in the sense  that it gives a simple summation with a product of two binomials. 
\begin{theorem} The following identity holds true
\begin{align} \label{917}
\binom{n}{k}_{[s]}=\sum_{j}(-1)^{j}\binom kj\binom{n+k-sj}{2k}.
\end{align}
\end{theorem} 
\begin{proof}
By Theorem \ref{aqw} we have
\begin{align*}
\sum_{n\geq0}\binom{n}{k}_{[s]}x^{n}&=\frac{x^{k}(1+x+x^{2}+\cdots+x^{s-1})^{k}}{(1-x)^{k+1}}\\
&=x^{k}(1-x^{s})^{k}\frac{1}{(1-x)^{2k+1}}\\
&=x^{k}\sum_{j}(-1)^{j}\binom kj x^{js}\sum_{i}\binom{i+2k}{2k}x^{i}\\
&=\sum_{n}\sum_{i+sj=n}(-1)^{j}\binom kj\binom{i+2k}{2k}x^{n+k}\\
&=\sum_{n}\sum_{j}(-1)^{j}\binom kj\binom{n-sj+k}{2k}x^{n}.\\
\end{align*}
Identity (\ref{917}) is a dual version of Relation (\ref{js})
\end{proof}
Now we give an identity for $\binom{n}{k}_{[s]}$ dual to Relation (\ref{bb1}).
\begin{theorem}\label{b4} For $w=\exp{(2i\pi/s)}$ we have
\begin{align*}
\binom{n}{k}_{[s]}=\sum_{j}\binom{n-j}{k}(-1)^{j}\sum_{k_{1}+k_{2}+\cdots+k_{s-1}=j}\binom{k}{k_{1}}\binom{k}{k_{2}}\cdots\binom{k}{k_{s-1}}\times w^{-\sum_{r=1}^{s-1}rk_{r}}.
\end{align*}
\end{theorem}
\begin{proof}By Theorem \ref{aqw}, we have 
\begin{align*}
&\sum_{n\geq0}\binom{n}{k}_{[s]}x^{n}\\
&=\frac{x^{k}}{(1-x)^{k+1}}\prod_{j=1}^{s-1}(x-w^{j})^{k}\\
&=\sum_{n\geq0}\binom{n}{k}x^{n}\sum_{k_{1}}(-1)^{k-k_{1}}\binom{k}{k_{1}}w^{k-k_{1}}x^{k_{1}}
\sum_{k_{2}}(-1)^{k-k_{2}}\binom{k}{k_{2}}w^{2(k-k_{2})}x^{k_{2}}\cdots\\
&\cdots\sum_{k_{s-1}}(-1)^{k-k_{s-1}}\binom{k}{k_{s-1}}w^{(s-1)(k-k_{s-1})}x^{k_{s-1}}\\
\end{align*}
\begin{align*}
&=\sum_{n\geq0}\binom{n}{k}x^{n}\sum_{j\geq0}\sum_{k_{1}+k_{2}+\cdots+k_{s-1}=j}\binom{k}{k_{1}}\binom{k}{k_{2}}\cdots\binom{k}{k_{s-1}}\times(-1)^{(s-1)k-j}{w^{\sum_{r=1}^{s-1}r(k-k_{r})}x^{j}}\\
&=\sum_{n\geq0}\sum_{j}\binom{n-j}{k}\sum_{k_{1}+k_{2}+\cdots+k_{s-1}=j}\binom{k}{k_{1}}\binom{k}{k_{2}}\cdots\binom{k}{k_{s-1}}\times(-1)^{(s-1)k-j}w^{\sum_{r=1}^{s-1}r(k-k_{r})}x^{n}.\\
&=\sum_{n\geq0}\sum_{j}\binom{n-j}{k}\sum_{k_{1}+k_{2}+\cdots+k_{s-1}=j}\binom{k}{k_{1}}\binom{k}{k_{2}}\cdots\binom{k}{k_{s-1}}(-1)^{j}w^{-\sum_{r=1}^{s-1}rk_{r}}x^{n},\\
&\text{This yields the result}.
\end{align*}
\end{proof}
\begin{remark}We can also deduce Theorem \ref{b4} using Theorem \ref{azer} and Identity (\ref{bb1}).
\end{remark}
Using the  generating function we obtain the following beautiful nested relation.\\
\begin{theorem} The terms of the $s$-Pascal triangle satisfy the following identity
\begin{align*}
\binom{n}{k}_{[s]}=\sum_{j_{1},j_{2},\ldots,j_{s-1}}\binom{k}{j_{1}}\binom{j_{1}}{j_{2}}\cdots\binom{j_{s-2}}{j_{s-1}}\binom{n-k-\sum_{i=1}^{s-2}j_{i}}{j_{s-1}}\times(2)^{j_{1}}(3/2)^{j_{2}}\cdots(s/s-1)^{j_{s-1}}.
\end{align*}
\end{theorem}
\begin{proof}We have
\begin{align*}
&\sum_{n\geq0}\sum_{j_{1},j_{2},\ldots,j_{s-1}}\binom{k}{j_{1}}\binom{j_{1}}{j_{2}}\cdots\binom{j_{s-2}}{j_{s-1}}\binom{n-k-\sum_{i=1}^{s-2}j_{i}}{j_{s-1}}(2)^{j_{1}}(3/2)^{j_{2}}\cdots(s/s-1)^{j_{s-1}}x^{n}\\
&=x^{k}\sum_{j_{1}}\binom{k}{j_{1}}(2x)^{j_{1}}\sum_{j_{2}}\binom{j_{1}}{j_{2}}(3x/2)^{j_{2}}\cdots\sum_{j_{s-2}}\binom{j_{s-3}}{j_{s-2}}(x(s-1)/(s-2))^{j_{s-2}}\times\\
&\times\sum_{j_{s-1}}\binom{j_{s-2}}{j_{s-1}}(s/s-1)^{j_{s-1}}\sum_{n\geq0}\binom{n-k-\sum_{i=1}^{s-2}j_{i}}{j_{s-1}}x^{n-k-\sum_{i=1}^{s-2}j_{i}}\\
&=\frac{x^{k}}{1-x}\sum_{j_{1}}\binom{k}{j_{1}}(2x)^{j_{1}}\sum_{j_{2}}\binom{j_{1}}{j_{2}}(3x/2)^{j_{2}}\cdots\sum_{j_{s-2}}\binom{j_{s-3}}{j_{s-2}}(x(s-1)/(s-2))^{j_{s-2}}\times\\
&\times\sum_{j_{s-1}}\binom{j_{s-2}}{j_{s-1}}\left( \frac{(s/(s-1))x}{1-x}\right)^{j_{s-1}}\\
&=\frac{x^{k}}{1-x}\sum_{j_{1}}\binom{k}{j_{1}}(2x)^{j_{1}}\sum_{j_{2}}\binom{j_{1}}{j_{2}}(3x/2)^{j_{2}}\cdots\\
&\cdots\sum_{j_{s-2}}\binom{j_{s-3}}{j_{s-2}}\left(\frac{x(s-1)/(s-2)+x^{2}/(s-2)}{1-x}\right)^{j_{s-2}}\\
&=\frac{x^{k}}{1-x}\sum_{j_{1}}\binom{k}{j_{1}}(2x)^{j_{1}}\sum_{j_{2}}\binom{j_{1}}{j_{2}}(3x/2)^{j_{2}}\cdots\\
&\cdots\sum_{j_{s-3}}\binom{j_{s-4}}{j_{s-3}}\left(\frac{(x(s-2)+x^{2}+x^{3})/(s-3)}{1-x}\right)^{j_{s-3}}\\
\ &\vdots \\
&=\frac{x^{k}}{1-x}\sum_{j_{1}}\binom{k}{j_{1}}\left(\frac{2x+x^{2}+\cdots +x^{s-1}}{1-x}\right)^{j_{1}}\\
&=\frac{x^{k}(1+x+x^{2}+\cdots+x^{s-1})^{k}}{(1-x)^{k+1}}.
\end{align*}
\end{proof}
\section{The $q$-analogue of the quasi s-Pascal triangle}
In this section we define  the $q$-analogue of the quasi $s$-Pascal triangle, we denote by ${n \brack k}_{[s]}$ these coefficients; for that, we give an explicit formula and generating function of ${n \brack k}_{[s]}$, and finally we propose a $q$-deformation for the $s$-bonacci sequence.
\begin{definition}\label{02} We define the $q$-quasi-bi$^{s}$nomial coefficient, according to Relation (\ref{aabc}), as
\begin{align}\label{bn1}
{n \brack k}_{[s]}={n-1 \brack k}_{[s]}+\sum_{j=1}^{s}q^{n-j}{n-j \brack k-1}_{[s]},\hspace*{0.6cm}
\end{align}
or equivalently
\begin{align}\label{bn2}
{n \brack k}_{[s]}=q^{k}{n-1 \brack k}_{[s]}+\sum_{j=1}^{s}q^{(k-1)j}{n-j \brack k-1}_{[s]}.
\end{align}
We use the convention ${0 \brack 0}_{[s]}=1 $ and ${n \brack k}_{[s]}=0$ for $k \notin \lbrace 0,\ldots,n \rbrace$.
\end{definition}
\begin{remark}
For $s=1$ we obtain Relations (\ref{03}) and (\ref{04}) respectively.
\end{remark}
The generating function of ${n \brack k}_{[s]}$ is given by
\begin{theorem}\label{79} Let $\mathbb{F}_{k}(x):=\sum_{n\geq0}{n \brack k}_{[s]}x^{n}$ the generating function of the $q$-quasi bi$^{s}$nomial coefficient, then 
\begin{align*}
\mathbb{F}_{k}(x)=\frac{x^{k}q^{\binom{k}{2}}\prod_{j=0}^{k-1}(1+q^{j}x+(q^{j}x)^{2}+\cdots+(q^{j}x)^{s-1})}{\prod_{j=0}^{k}(1-q^{j}x)}.
\end{align*}
\end{theorem}
\begin{proof}
We have $\mathbb{F}_{k}(x)=\sum_{n\geq0}{n \brack k}_{[s]}x^{n}$, then using Relation (\ref{bn1})
\begin{align*}
\mathbb{F}_{k}(x)&=\sum_{n\geq0}{n-1 \brack k}_{[s]}x^{n}+\sum_{n\geq0}\sum_{j=1}^{s}q^{n-j}{n-j \brack k-1}_{[s]}x^{n}\\
&=x\sum_{n\geq0}{n \brack k}_{[s]}x^{n}+(x+x^{2}+\cdots+x^{s})\sum_{n\geq0}{n \brack k-1}_{[s]}(qx)^{n},\\
\end{align*}
thus
\begin{align}\label{2210}
\frac{(1-x)\mathbb{F}_{k}(x)}{(x+x^{2}+\cdots+x^{s})}=\sum_{n\geq0}{n \brack k-1}_{[s]}(qx)^{n},
\end{align}
on the other side, using Relation (\ref{bn1}) again
\begin{align*}
&\frac{(1-x)\mathbb{F}_{k}(x)}{(x+x^{2}+\cdots+x^{s})}\\
&=\sum_{n\geq0}{n-1 \brack k-1}_{[s]}(qx)^{n}+\sum_{n\geq0}q^{n-1}{n-1 \brack k-2}_{[s]}(qx)^{n}+\cdots+\sum_{n\geq0}q^{n-s}{n-s \brack k-2}_{[s]}(qx)^{n}\\
&=qx\sum_{n\geq0}{n \brack k-1}_{[s]}(qx)^{n}+(qx+(qx)^{2}+\cdots+(qx)^{s})+\sum_{n\geq0}{n \brack k-2}_{[s]}(q^{2}x)^{n},
\end{align*}
and by Relation (\ref{2210}) we obtain
\begin{align}
\frac{(1-x)(1-qx)\mathbb{F}_{k}(x)}{(x+x^{2}+\cdots+x^{s})(qx+(qx)^{2}+\cdots+(qx)^{s})}=\sum_{n\geq0}{n \brack k-2}_{[s]}(q^{2}x)^{n},
\end{align}
we repeat the process and get\\
\begin{align*}
\frac{(1-x)(1-qx)\cdots(1-q^{k-1}x)\mathbb{F}_{k}(x)}{\prod_{j=0}^{k-1}(q^{j}x+(q^{j}x)^{2}+\cdots+(q^{j}x)^{s}}&=\sum_{n\geq0}{n \brack 0}_{[s]}(q^{k}x)^{n}\\
&=\frac{1}{(1-q^{k}x)},
\end{align*}
and finally we conclude to the result.\\
\end{proof}
\begin{remark} With the same method we can also prove Theorem \ref{79}, using Relation (\ref{bn2}).
\end{remark}
The following result establishes the explicit formula of the $q$-quasi-bi$^{s}$nomial coefficient
\begin{theorem}\label{bn} The coefficient ${n \brack k}_{[s]}$ satisfy
\begin{align}
{n \brack k}_{[s]}=\sum_{j}{n-j \brack k}{k \brack j}^{(s-1)}.
\end{align}
\end{theorem}
\begin{proof} We have
\begin{align*}
\sum_{n\geq0}\sum_{j}{n-j \brack k}{k \brack j}^{(s-1)}x^{n}&=\sum_{j}{k \brack j}^{(s-1)}x^{j}\sum_{n\geq0}{n-j \brack k}x^{n-j}\\
&=\frac{x^{k}q^{\binom{k}{2}}\prod_{j=0}^{k-1}(1+q^{j}x+(q^{j}x)^{2}+\cdots+(q^{j}x)^{s-1})}{\prod_{j=0}^{k}(1-q^{j}x)},
\end{align*}
this last equality comes from Relation (\ref{2t}) and Relation (\ref{gh}).
\end{proof}
Cigler and Carlitz propose the $q$-analogue of the Fibonacci sequence see \cite{CL1,JC1}, the following result establish the recurrence relation for the $q-$analogue of the $s-$bonacci sequence, this generalize s Theorem \ref{aqwd}.
\begin{theorem} Let $\mathbb{T}_{n+1}^{(s)}(x):=\sum_{k}{n-k \brack k}_{[s]}x^{k}$ for $n\geq 0$ and $\mathbb{T}_{0}^{(s)}(x)=0$ then\\
\begin{align}\label{es}
\mathbb{T}_{n+1}^{(s)}(x)=\mathbb{T}_{n}^{(s)}(x)+x\sum_{j=1}^{s}q^{n-j-1}\mathbb{T}_{n-j}^{(s)}(x/q),
\end{align}
and
\begin{align}\label{es1}
\mathbb{T}_{n+1}^{(s)}(x)=\mathbb{T}_{n}^{(s)}(xq)+x\sum_{j=1}^{s}\mathbb{T}_{n-j}^{(s)}(xq^{j}).
\end{align}
\end{theorem}
\begin{proof} For Relation (\ref{es}), we have 
\begin{align*}
\mathbb{T}_{n+1}^{(s)}(x)=\sum_{k}{n-k \brack k}_{[s]}x^{k},
\end{align*}
then by Relation (\ref{bn1}), we have
\begin{align*}
\mathbb{T}_{n+1}^{(s)}(x)=\sum_{k}{n-k-1 \brack k}_{[s]}x^{k}+\sum_{k}{n-k-1 \brack k-1}_{[s]}x^{k}q^{n-k-1}+\cdots\hspace*{1.5cm}\\
\cdots+\sum_{k}{n-k-s \brack k-1}_{[s]}x^{k}q^{n-k-s}\\
=\sum_{k}{n-k-1 \brack k}_{[s]}x^{k}+x\sum_{k^{'}}{n-k^{'}-2 \brack k^{'}}_{[s]}x^{k^{'}}q^{n-k^{'}-2}+\cdots \hspace*{0.9cm}\\
\hspace*{0.7cm}\cdots+x\sum_{k^{'}}{n-k^{'}-s-1 \brack k^{'}}_{[s]}x^{k^{'}}q^{n-k^{'}-s-1}\\
=\mathbb{T}_{n}^{(s)}(x)+x\sum_{j=1}^{s}q^{n-j-1}\mathbb{T}_{n-j}^{(s)}(x/q).\hspace*{5cm}
\end{align*}
\end{proof}
The proof is the same for the  Relation (\ref{es1}), we use Relation (\ref{bn2}).

\end{document}